\documentclass[11pt]{amsart}
\usepackage{eurosym}
\usepackage{amsmath,amsthm,amsfonts,amssymb,amscd}
\usepackage{verbatim}
\usepackage{graphicx}

\numberwithin{equation}{section}
\theoremstyle{plain}
\newtheorem{thm}{Theorem}[section]
\newtheorem*{thm-no}{Theorem}
\newtheorem{prop}[thm]{Proposition}
\newtheorem{lemma}[thm]{Lemma}
\newtheorem{clly}[thm]{Corollary}
\newtheorem{maintheorem}{Theorem}

\theoremstyle{definition}
\newtheorem{remark}[thm]{Remark}

\newtheorem{defi}[thm]{Definition}

\newcommand{\real}{{\mathbb R}}

\newcommand{\natu}{{\mathbb N}}
\newcommand{\QQ}{{\mathbb Q}}
\newcommand{\bI}{{\mathbb I}}

\def \cE {{\mathcal E}}

\def \cN {{\mathcal N}}

\def \cR {{\mathcal R}}

\def \cU {{\mathcal U}}

\def \Ha {H(\alpha)}

\newcommand{\Var}{\textrm{Var}}

\newcommand{\Lip}{\textrm{Lip}}
\email[S. Galatolo]{galatolo@dm.unipi.it}
\urladdr{http://users.dma.unipi.it/galatolo/}
\thanks{ M. J. P. was partially supported by CNPq,
  PRONEX-Dyn.Syst., FAPERJ.\\
  I. N. was partially supported by CNPq, FAPERJ, University of Uppsala and KAW grant 2013.0315.}
\email{pacifico@im.ufrj.br}
\email{nisoli@im.ufrj.br}
\input{tcilatex}

\begin{document}
\title[Correlations $\&$ logarithm law for contracting Lorenz attractors .]{Decay of
correlations, quantitative recurrence and logarithm law for contracting Lorenz attractors.}
\author{Stefano Galatolo}
\address[S.G.]{Dipartimento di Matematica, Largo Pontevcorvo Pisa}
\author{Isaia Nisoli}
\address[I.N.]{Instituto de Matem\'atica, Universidade Federal do Rio de
Janeiro\\
C. P. 68.530, 21.945-970 Rio de Janeiro}
\author{Maria Jose Pacifico}
\address[M.J.P.]{Instituto de Matem\'atica, Universidade Federal do Rio de
Janeiro\\
C. P. 68.530, 21.945-970 Rio de Janeiro}
\date{}

\begin{abstract}
In this paper we prove that a class of skew products maps with non
uniformly hyperbolic base has exponential decay of correlations. We apply this to
obtain a logarithm law for the hitting time associated to a \emph{
contracting Lorenz} attractor at all the points having a well
defined local dimension, and a quantitative recurrence estimation.
\end{abstract}

\maketitle


\section{Introduction}

The term \emph{statistical properties} of a dynamical system $F:M\to M$,
where $M$ is a measurable space and $F$ a measurable map, refers to the long time
behavior of large sets of trajectories of the system. It is well known that
this relates to the properties of the transfer operator, a linear operator
associated to the dynamics that embodies how the measures evolve under 
the action of the system.

Statistical properties are often a better object to be studied than
pointwise behavior. In fact, the future behavior of initial data can be
unpredictable, but statistical properties are often regular and their
description simpler. Suitable results can be established in many cases, to
relate the evolution of measures with that of large sets of points (ergodic
theorems, large deviations, central limit, logarithm law, etc...).

In this paper we take the point of view of studying the evolution of
measures and its speed of convergence to equilibrium to understand the
statistical properties of a class of dynamical systems.

We consider a \emph{contracting Geometric
Lorenz flow } and its perturbations \cite{Ro93}, a relative of the Lorenz flow,
but strictly non-uniformly hyperbolic;
following the cited paper, we have that if the
Geometric Flow admits an attractor and we take a one parameter 
family of perturbations, parametrized
by a parameter $a$, there exists $0< a_0< 1$ and $E\subset [0,a_0)$ a set of parameters
called the \textbf{Rovella parameters}, of positive Lebesgue measure,
that admit an attractor.

We prove that the decay of correlations for the map $F_a$, induced
on a Poincar\'{e} section for a Rovella parameter $a\in E$
is exponential with respect to Lipschitz observables:
this permits us to study some of the statistical features of the map and the flow.

We follow the line of \cite{VGP2014} exploiting the fact
that the system has an invariant contracting fibration and hence can be seen as
a skew product whose base transformation has exponential convergence to
equilibrium (a measure of how fast iterates of the Lebesgue measure converge
to the physical measure), proving the following.

\begin{maintheorem}\label{teoA}
For all Rovella parameters $a\in E$, the 
two dimensional map $F_a$ associated to the flow has exponential 
decay of correlations with respect to
Lipschitz observables : there exists $C,\Lambda \in \mathbb{R}^{+},~\Lambda <1$
such that 
\begin{equation*}
\left\vert \int f\cdot (g\circ F_a^{n})\,d\mu -\int g\,d\mu \int f\,d\mu
\right\vert \leq C\Lambda ^{n}\cdot ||g||_{Lip}\cdot||f||_{Lip}.
\end{equation*}
\end{maintheorem}

The rapid decay of correlations give, as a consequence, a quantitative
recurrence estimation and an estimation for the scaling behavior of the time which is needed to hit small targets (Section \ref{s.loglaw}).

Next, we establish a logarithm law for the dynamics of the
flow $X_a^t$, for $a \in E$.  This is a relation between hitting time to small
targets and the local dimension of the invariant measure we
consider.  Let us consider the family of balls 
$B_{r}(x_0)$, with center $x_0$ and radius  $r$, and let us denote the time needed for
the orbit of a point $x$ to enter in $B_{r}(x_0)$ by
\begin{equation*}
\tau^{F_a}_r(x,x_0):=\min \{n\in \mathbb{N}^{+}:F_a^{n}(x)\in B_{r}(x_0)\}.
\end{equation*}
A logarithm law is the statement that as $r\to0 $ the hitting time scales like $1/\mu(B_r )$. When $x=x_0$ let us denote
 \begin{equation*}
\tau^{F_a}_r(x_0):=\tau^{F_a}_r(x_0,x_0).
\end{equation*}
This is an indicator for the recurrence time at the point $x_0$. For this indicator too, it is possible to show quantitative recurrence statements showing that as  $r\to0 $ the recurrence time scales like $1/\mu(B_r )$.
To express this more precisely let us consider the local dimensions of a measure $\mu$
\begin{equation}\label{eq:localdim}
  \overline{d}_{\mu}(x_0)=
  \underset{r\rightarrow 0}{\lim \sup }\frac{\log \mu (B_{r}(x_0))}{%
    \log (r)}
  \quad\text{and}\quad 
  \underline{d}_{\mu}(x_0)=\underset{r\rightarrow 0}{\lim \inf }\frac{\log
    \mu (B_{r}(x_0))}{\log (r)}
\end{equation}
representing the scaling rate of the measure of small balls
as the radius goes to $0$.  
When the above limits coincide
for $\mu $-almost every point, we 
 set $d_{\mu
}=\underline{d}_{\mu }(x)=\overline{d}_{\mu }(x)$\footnote{See Section \ref{s.loglaw} for more precise definitions on local dimension and the hitting/return times $\tau_r$.}.  



\begin{maintheorem}\label{teoB}
For the Rovella map $F_{a}$, $a \in E$, where
$E$ is the set of Rovella parameters, and $\mu_{F_a}$ is the invariant SBR 
measure for the map $F_a$.
For $\mu _{F_a}$ almost every $x_0$:
\begin{equation*}
\liminf_{r\rightarrow 0}\frac{\log \tau_{r}^{F_a}(x_0)}{-\log r}={
\underline{d}}_{\mu _{F_a}}(x_0),\quad \limsup_{r\rightarrow 0}\frac{\log \tau
_{r}^{F_a}(x_0)}{-\log r}={\overline{d}}_{\mu_{F_a}}(x_0).
\end{equation*}

Moreover, for each regular point $x_{0}\in \Sigma_a$ 
such that the local dimension of $\mu_{F_a}$ at $x_{0}$ $d_{\mu_{F_a}}(x_{0})$ exists, it holds 
\begin{equation*}
\lim_{r\rightarrow 0}\frac{\log \tau _{r}^{F_a}(x,x_{0})}{-\log r}=d_{\mu
_{F_a}}(x_{0})
\end{equation*}
for $\mu _{F_a}$-almost each $x\in \Sigma_a.$
\end{maintheorem}

The logarithm law for the two dimensional map and the integrability of the
return time with respect to $\mu_{F_{a}}$ imply a logarithm law for the 
hitting time and recurrence time of the flow,
in a way similar to what was done in \cite{galapacif09}.  Let $x, x_{0}\in
\Lambda$ and
\begin{equation}\label{circlen}
\tau _{r}^{X_a^t}(x,x_{0})=\inf \{t\geq 0 | X_a^{t}(x)\in B_r(x_{0})\}
\end{equation}
be the time needed for the $X_a^t$-orbit of a point $x$ to enter
for the {\em first time} in a ball $B_{r}(x_{0})$. The
number $\tau_{r}^{X_a^t}(x,x_{0})$ is the \emph{hitting time}
associated to the flow $X_a^t$ and target $B_r(x_0)$. 
By this we can also consider a quantitative recurrence indicator $\tau_{r}^{X_a^t}(x_{0})$ as done before (see Definition \ref{eq-2}).


\begin{maintheorem}\label{teoC}
If $X_a^t$ is a contracting Lorenz flow with $a\in E$, where
$E$ is the set of Rovella parameters, and $\mu_{X_a}$ is an invariant SBR 
measure for the flow $X_a^t$, then for each regular point $x_{0}\in \mathbb{R}^{3}$ 
such that $d_{\mu_{X_a}}(x_{0})$ exists, it holds 
\begin{equation*}
\lim_{r\rightarrow 0}\frac{\log \tau _{r}^{X_a^t}(x,x_{0})}{-\log r}=d_{\mu
_{X_a}}(x_{0})-1
\end{equation*}
for $\mu _{X_a}$-almost each $x\in \mathbb{R}^{3}.$ While for $\mu _{X_a}$
almost every $x_0$
\begin{equation*}
\liminf_{r\rightarrow 0}\frac{\log \tau _{r}^{X_a}(x_0)}{-\log r}={
\underline{d}}_{\mu _{X_a}}(x_0)-1,\quad \limsup_{r\rightarrow 0}\frac{\log \tau
_{r}^{X_a}(x_0)}{-\log r}={\overline{d}}_{\mu _{X_a}}(x_0)-1.
\end{equation*}
\end{maintheorem}

This extend the results of \cite{galapacif09} to the contracting Geometric Lorenz flow and its perturbations.

\subsection{Organization of the text}
This paper is organized as follows. In Section \ref{s.flow} we introduce the main object of this article, the contracting  Lorenz flow. In Section \ref{sec:one_dimensional} we explicit the main properties of the one dimensional map associated to a contracting Lorenz flow. In Section \ref{s-general} we recall some general results 
on convergence and correlation decay for
skew-products with contracting fibers
from \cite{VGP2014}, that will be used to prove
Theorem \ref{teoA}. In Section \ref{sec:p-boundeddecay}  we show how to extend
some results about decay
of correlations and convergence to equilibrium for H\"{o}lder
observables to generalized bounded variation observables. This extension is necessary to apply the general results of the previous section to our case.
In Section \ref{2p1} we establish exponential decay of correlations with respect to Lipschitz observables for the two dimensional  map associated to a contracting Lorenz flow and prove Theorem \ref{teoA}.
In Section \ref{s.loglaw}  we show some consequences of the decay of correlations
proved above as  hitting time and quantitative
recurrence estimations.
Finally, in Section \ref{sec:linearization} we explain a result about linearization and properties of the Poincar\'{e} map associated to a contracting Lorenz flow needed along the paper.

\section{The contracting Lorenz flow}\label{s.flow}
In this section we present the family of dynamical systems first studied at \cite{Ro93}, which are the object of our paper. 

The starting point for its 
definition is the geometric contracting  Lorenz Flow, a strict relative of the geometric Lorenz 
system, in which the uniformly expanding direction is replaced by a strict nonuniformly
expanding direction.

We describe informally 
this construction 
following \cite{PT2010}.
Let $(\dot x, \dot y, \dot
z)=(\lambda_1 x,\lambda_2 y, \lambda_3 z)$ be a vector field in the cube $[-1,1]^3$, with a singularity at the origin $(0,0,0)$. 
Suppose the eigenvalues $\lambda_i$, $1\le i\le 3$ satisfy the relations
\begin{equation}
 \label{eq:eigenvalues}
-\lambda_2 > - \lambda_3> \lambda_1 > 0,\quad r= - \frac{\lambda_2}{\lambda_1},\,\,
s=- \frac{\lambda_3}{\lambda_1},\quad r>s +3.
\end{equation}

It is worth to remark that $\lambda_1 + \lambda_3 < 0$ in the contracting case while in the definition of the usual geometric Lorenz flow the condition is $\lambda_1 + \lambda_3 > 0$.
The condition $r> s+3$ is used in \cite{Ro93} to guarantee the existence of a $C^3$
uniformly contracting stable foliation for the Poincar\'{e} first return map of perturbations of the geometric contracting Lorenz flow. 

Let $\Sigma^{-}=\{(x,y,1)\mid -1/2\leq x\leq 0, |y|\leq 1/2\}, \Sigma^{+}=\{(x,y,1)\mid 0\leq x\leq 1/2, |y|\leq 1/2\}$ and $\Sigma=\Sigma^{+}\cup \Sigma^{-}$.
In figure \ref{L3Dcusp} we can see the behaviour of the field near the origin, with some computations, it is possible to see that the flow 
reaches the transverse section $x=1$ (a similar reasoning works for $x=-1$) obeying the following law:
\begin{equation}\label{eq:cross_sec}
\tilde{F}(x,y,1)=(1,yx^r,x^s)
\end{equation}

\begin{figure}[h]
\begin{center}
\includegraphics[width=7cm]{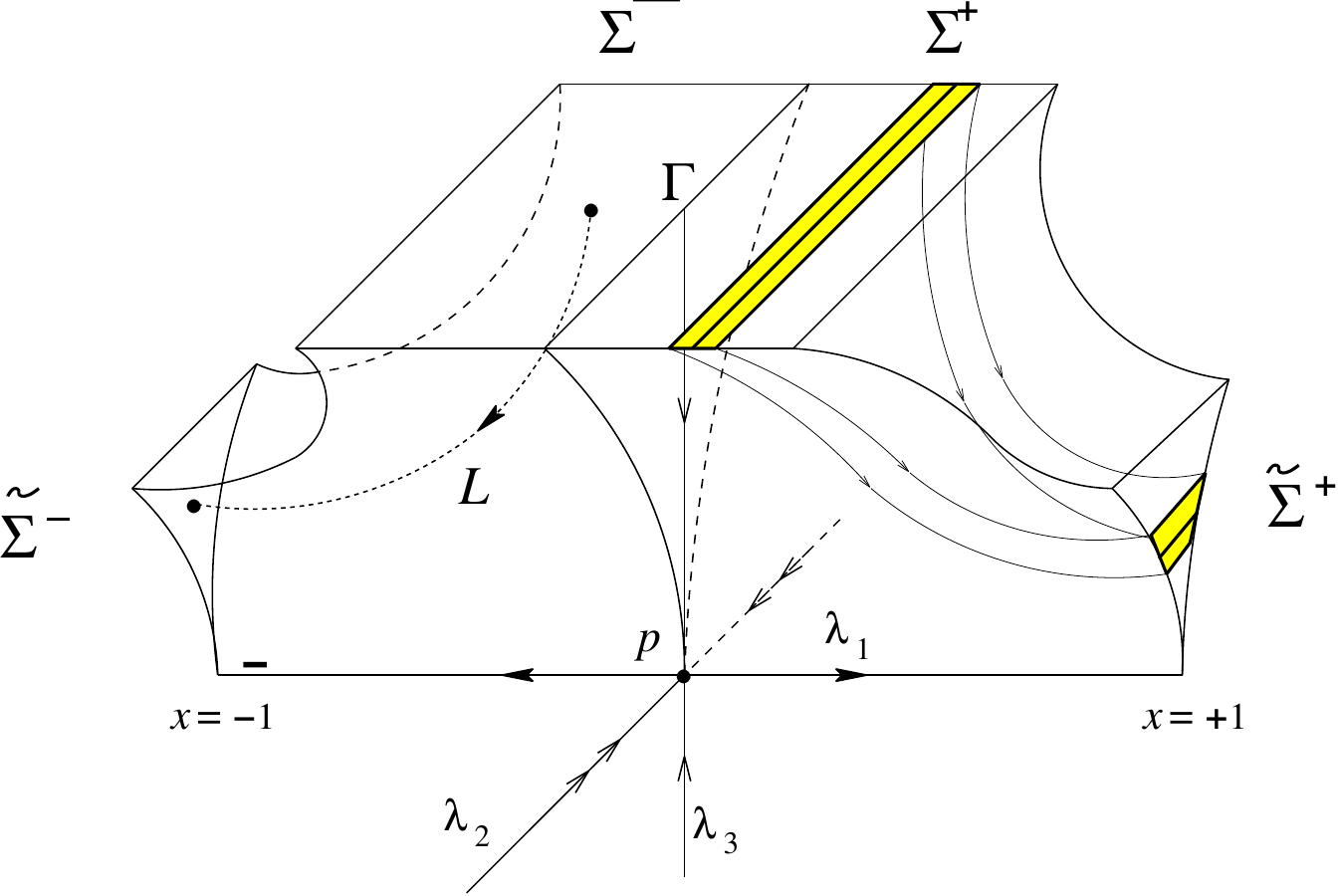}
\end{center}
\caption{\label{L3Dcusp}Behaviour near the origin.}
\end{figure}

Outside the cube, as in the Geometric Lorenz case, to imitate the random turns of a regular orbit around the origin and obtain
a butterfly shape for our flow,  we let the flow return to the cross section $\Sigma$ through a
flow described by a suitable composition of a rotation $R_\pm$, an expansion $E_{\pm\rho}$ and a translation $T_\pm$.
The resulting effect of the flow outside the cube when we arrive on $\Sigma$ may be represented by a rotation and the expansion which have the form:
\[
R_\pm(x,y,z)=
\begin{pmatrix}
0 & 0 & \pm 1 \\
0 & 1 & 0 \\
\pm1 &0 &0 
\end{pmatrix}, \quad
E_{\pm,\rho}(x,y,z)=
\begin{pmatrix}
\rho & 0 & 0 \\
0 & 1 & 0 \\
0 &0 &1 
\end{pmatrix},
\]
$\rho$ is such that $\rho\cdot (1/2)^s \leq 1$. This condition is an hypothesis on the behavior of the vector field outside a neighborhood of the origin, and insures that the image of the map is contained in $\Sigma$.
Further, we can also take $\rho$ sufficiently small to guarantee that the contraction along the stable foliation is stronger than $\rho$.
The condition on the eigenvalues expressed in equation (\ref{eq:eigenvalues}) gives the necessary condition to obtain this contraction, see \cite[Remarks, page 240]{Ro93}.

The translations $T_\pm$ are chosen in such a way that the unstable direction starting from the origin is sent to the boundary of $\Sigma$ and 
the images of $\tilde{\Sigma}_{\pm}$ are disjoint. 

It is possible to find a flow that realizes this construction, as it is described in \cite{galapacif09,PT2010}.
Composing the expression in \eqref{eq:cross_sec} with $R_\pm$,$E_{\pm\rho}$ and $T_\pm$
we write an explicit formula for $F_0:\Sigma\to \Sigma$, the Poincar\'{e} first return map
of the geometric contracting Lorenz flow on the section $\Sigma$: 
\begin{align}\label{e:Rovelladimension2}
F_0(x,y)&=(T_0(x),G_0(x,y))\\
T_0(x)=\left\{\begin{array}{cc}\nonumber
-\rho |x|^s+1/2 & x>0\\
\rho |x|^s-1/2 & x<0
\end{array}\right.&, 
G_0(x,y)=\left\{\begin{array}{cc}
y|x|^r+c_0 & x>0\\
-y|x|^r+c_1 & x<0
\end{array}\right.,
\end{align}
where $c_0,c_1$ are real numbers depending on the choice of the translations $T_{\pm}$,
$r$ and $s$ are as in (\ref{eq:eigenvalues}) and $\rho \leq (1/2)^{-s}$. 
\vspace{0.2cm}

It is proved in \cite[Item 4, page 240]{Ro93} that the map $T_0$
satisfies the following properties:
\begin{itemize}
\item[(a)] $T_0$ is onto and piecewise $C^{3}$, with two branches. The order\footnote{we say that $f(x)$ is $O(g(x))$ at $x=x_0$ if there exists $M,\delta$ such that $|f(x)|\leq M |g(x)|$ when $0<|x-x_0|<\delta$.}
of $T_0^{\prime }(x)=O(x^{s-1})$  
at $x=0$ where $s=-\frac{\lambda _{3}}{\lambda_{1}}$ and $s-1>0$,

\item[(b)] $T_0$ has a discontinuity at $x=0$, $T_0(0^{+})=1/2$, $T_0(0^{-})=-1/2$,

\item[(c)] $T_0^{\prime }(x)<0$ for every $x\neq 0$,

\item[(d)] $\max_{x>0}T_0^{\prime }(x)=T_0^{\prime }(1/2)$ and $
\max_{x<0}T_0^{\prime }(x)=T_0^{\prime }(-1/2)$.
\end{itemize}

There exist $\rho\leq (1/2)^{-s}$ such that also the following hypothesis are satisfied:
\begin{itemize}
\item[(e)] The points $1/2$ and $-1/2$ are preperiodic repelling for $T_0$.
\item[(f)] The map $T_0$ has negative schwarzian derivative.
\end{itemize}

Under these hypothesis, Rovella establishes that $X_0$ has an attractor $\Lambda_0$.

\begin{remark}
If $\rho=(1/2)^{-s}$, then the map $T_0$ is topologically conjugated to the doubling map and it
is called, in the literature, a ``full Rovella map'' \cite{PT2010}. For a ``full Rovella map'', 
the existence of an attractor and of an a.c.i.m. for $T_0$ are easily proved.

\end{remark}

Next, to announce the result of \cite{Ro93} we are interested in, let us recall the definition of an attractor and stability in the measure theoretical sense.
\begin{definition}
Let $X$ be a vector field, with associated flow $X^t$.
A set $\Lambda$ is an attractor for $X$ if it is compact, 
invariant under $X^t$, transitive (i.e., contains a dense orbit)
and it has a compact neighborhood $U$, called {\em{local basin}} of
$\Lambda$, such that $\Lambda=\bigcap_{t\geq 0} X^t(U)$.
\end{definition}

\begin{figure}[tbp]
\reflectbox{\includegraphics[width=30mm,height=30mm]{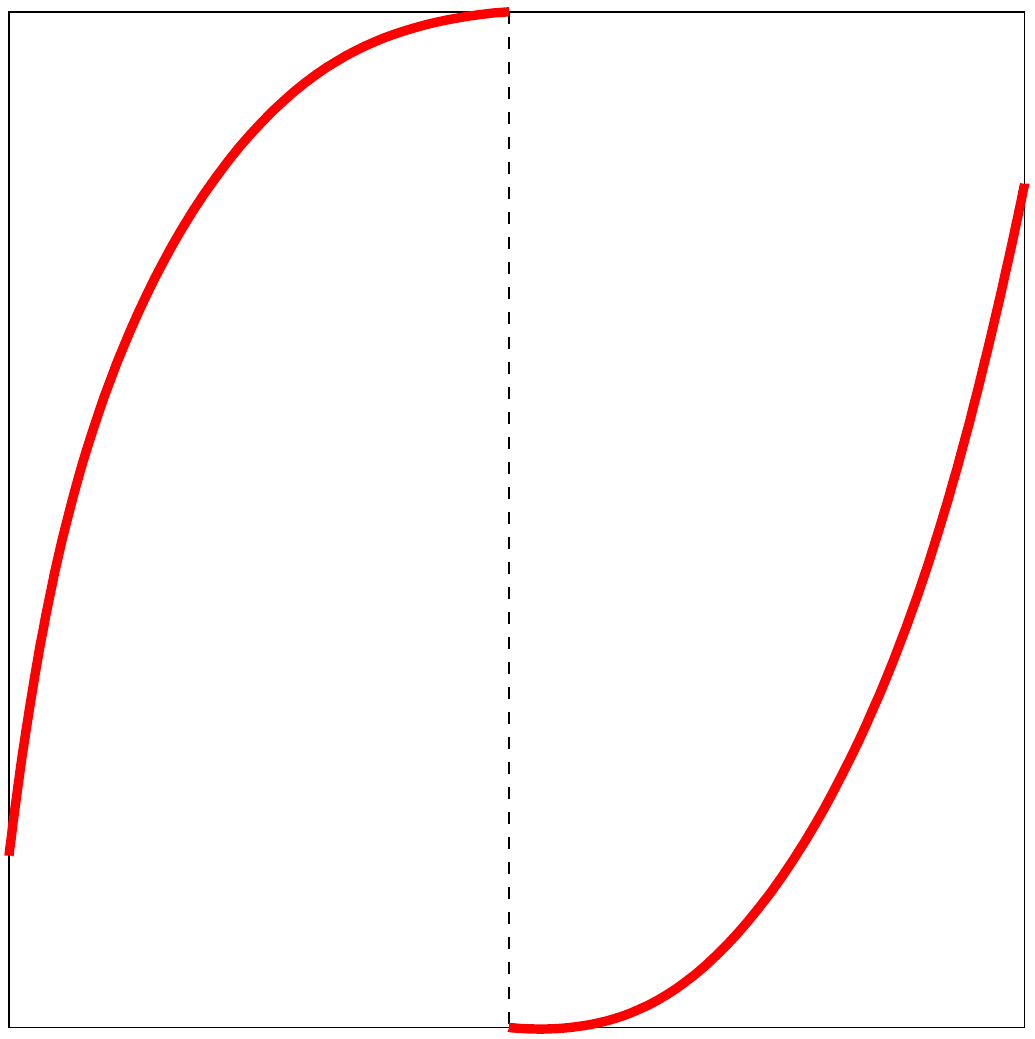}}
\caption{A $1$-dimensional contracting Lorenz map}
\label{Fig0}
\end{figure}

 \begin{definition}
Given a subset $S$ of a finite dimensional Riemannian manifold $\mathcal{M}$, we say that $x$ is a density point of $S$, if, denoting by $m$ the Lebesgue measure on $\mathcal{M}$, 
$B_r(x)$ the ball of radius $r$ and centered at $ x$, we have:
$$
\lim_{r\to 0}\frac{m(B_r(x)\cap S)}{m(B_r(x))}=1.
$$
\end{definition}
\begin{definition}
Given a subset $S$ of a Banach space $\mathcal{B}$, we say that $x \in   S$ is a point of $k$-dimensional full density of $S$ 
if there exists a $C^\infty$  submanifold $\mathcal{N}\subset \mathcal{B}$, containing $x$ and having codimension $k$, 
such that for every $k$- dimensional manifold $\mathcal{M}$ intersecting $\mathcal{N}$ transversally at $x$, then $x$ is a full density point of $S\cap \mathcal{M}$ in $\mathcal{M}$.
\end{definition}

\begin{definition}
An attractor $\Lambda$ of a flow $X^t\in C^{\infty}$ is $k$-dimensionally almost persistent if it has a local basin $U$ such that $X$ is a $k$-dimensional 
full density point of the set of flows $Y^t\in C^{\infty}$ for which $ \Lambda_Y = \cap_{t > 0}Y^t(U)$ is an attractor.
\end{definition}



In \cite{Ro93} it is proved that the attractor  $\Lambda_0$ constructed as above is $2$-dimensionally almost persistent in the $C^3$ topology, see item (b) of Theorem at page 235.



The proof of this result is similar to the proof of \cite{BC85} where it is proved that, for the map $f_a(x)=1-ax^2$ there exists a set of parameters of positive Lebesgue measure
for which $f_a$ has an absolutely continuous invariant measure.
The main step in the proof
is to exploit further the
hypotheses on the initial vector field that lead to the expression
at equation (\ref{e:Rovelladimension2}) for its Poincar\'e map.
This allows to reduce the problem to the analysis of
the one dimensional map induced by the flow. An auxiliary result is the following,
which proves the persistence of a $C^3$ stable foliation. 

\begin{theorem}\label{theo-foliation}{\cite[Proposition, page 241]{Ro93}}\/
There exists an open neighborhoods $\mathcal{U}$ of $X_0$ such that the flow of each 
$X \in \mathcal{U}$ admits a $C^3$ stable one dimensional foliation in $U$  that varies continuously with $X$.
\end{theorem}

By Frobenius Theorem, for each $X\in \mathcal{U}$ we can find a
transversal section to $X^t$ (and depending from $X^t$), 
near $\Sigma$ and consisting of
pieces of leafs of the stable foliation.

After a change of coordinates, for each $X\in \mathcal{U}$ the first return
map associated to $X$ can be written as
$$
F_X(x,y)=(T_X(x), G_X(x,y)).
$$
The one dimensional map $T_X$ induced by $F_X$ through the
foliation is $C^3$ in $x\neq 0$, $0$ is the discontinuity and the critical point. 
Furthermore, we can choose
coordinates on the transversal section such that
$T_X(0^+)= -1/2$, $T_X(0^-)=1/2$.

Let $\mathcal{U}$ be a $C^3$ neighbourhood of $X_0$ as
in Theorem \ref{theo-foliation} and define $\cN$ as
\begin{equation*}
\cN=\\
\{Y\in \mathcal{U} \mid \exists k^+, k^- \text{ so that }
T_Y^{k^+}(1/2),T_Y^{k^-}(-1/2) \text{ are periodic repelling}\}.
\end{equation*}
Note that if $\cU$ is small enough then $\cN$ is a codimension $2$ submanifold containing $X_0$.

We can now cite the main theorem proved in \cite{Ro93}:

\begin{theorem}[\cite{Ro93}]\label{thm:parameters}
Let $\mathcal{M}$ be a $2$-dimensional $C^3$-submanifold of $\cU$ intersecting $\cN$ transversally,
at $X_0$.
Let $\{X_a\}$ be a one parameter family of vector fields, contained in $\mathcal{M}$, such that
the functions $a\mapsto T_{X_a}(\mp1/2)$ have derivative $1$ at $a=0$.
Then there is a subset $E\subset (0,a_0)$ of parameters called the \textbf{Rovella parameters},
with $a_0$ close to $0$ and $0$ a full density point of $E$ such that
\[
\lim_{a\rightarrow 0}\frac{|E\cap (0,a)|}{a}=1, \,\,\mbox{with}\,\,
\Lambda _{X_{a}}=\cap_{t\geq 0}X_a^t(U)\,\,\mbox{an attractor.}
\]
This implies that $\Lambda_0$ is $2$-dimensionally almost persistent.
\end{theorem}

To ease notation, if $\{X_a\}$ is a one parameter family of vector
fields as above, we denote by
\[
F_a(x,y)=(T_a(x),G_a(x,y))
\]
the Poincar\'e map associated to $X_a$.

In the remaining of the paper, we restrict ourselves to the setting when $\{X_a\}$ is one of those one parameter families. 
We will denote the eigenvalues of $X_a$ at the singularity by $\lambda_{1,a},\lambda_{2,a},
\lambda_{3,a}$.

\begin{lemma}[\cite{Ro93}]
For all Rovella parameter $a\in E$ the 
induced $1$-dimensional map $T_a$ satisfies the following additional properties:
\begin{itemize}
\item[(C1)] $T'_a(x)=O(x^{s(a)-1})$ as $x\to 0$
where $s(a)=-\lambda_{3,a}/\lambda_{1,a}$. 

\item[(C2)] there is $\lambda_c> 1$ such that for all $a\in E$, the points $1/2 $ and $-1/2$ have Lyapunov exponents greater that $\lambda_c$: 
\begin{equation*}
(T_a^n)^{\prime }(\pm 1/2) > \lambda_c^n, \quad \mbox{ for all } \quad n\geq 1;
\end{equation*}

\item[(C3)] there is $\alpha >0$ such that for all $a\in E$ the \textit{basic
assumption} holds: 
\begin{equation*}
|T_a^{n-1}(\pm 1/2)| > e^{-\alpha n}, \quad \mbox { for all } \quad n\geq 1;
\end{equation*}

\item[(C4)] the forward orbits of the points $\pm 1/2$ under $T_a$ are dense
in $[-1/2,1/2]$ for all $a\in E$.

\item[(C5)] for all $a\in E$, $T_a$ has negative schwarzian derivative.
\end{itemize}
\end{lemma}

\begin{remark}
Item C1 depends on the possibility of linearizing the vector field $X_a$ near the origin;
while this hypothesis is not explicit in \cite{Ro93}, this is the reason why we assume
$X_a$ to be a $C^3$ one parameter family in the space of $C^{\infty}$ vector fields,
i.e., a $C^3$ map $\Xi:\mathbb{R}\supset(-\varepsilon,\varepsilon)\to C^{\infty}(\mathbb{R}^3)$, such that
$\Xi(0)=X_0$.

We refer to Section \ref{sec:linearization} for a discussion about the order of the derivatives of $T_a$
and of $G_a$.

The hypothesis of having a $C^k$ family in the space of $C^r$ vector fields is widely used
in the contracting Lorenz setting \cite{MoPaPu00,MPS06}.
\end{remark}

These properties have strong consequences on the statistical properties of 
the one dimensional $T_a$, as we will relate in Section \ref{sec:one_dimensional}.
In our work we show how the statistical properties of $T_a$
imply some statistical properties for the flow $X_a^t$ associated to a Rovella parameter
$a\in E$. 

\begin{remark}\label{r-propdeG}
The following properties of $X_a$ for all $a\in[0,a_0]$, used in our paper, 
follow from the properties of the linearization 
near the origin. We refer to Section \ref{sec:linearization}
for some explicit computations; let $r(a)=-\lambda_{2,a}/\lambda_{1,a}$ : 
\begin{enumerate}
\item $\frac{\partial G_a}{\partial y}(x,y)=O(x^{r(a)})$ as $x$ goes to $0$,
with $r(a)>s(a)+3$ which implies that $r(a)>3$,
\item the map $G_a$ is contracting along the leaves of the stable foliation, due to the 
fact that $\lambda_{2,a}$ is near $\lambda_2 < 0$,
\item the order of 
$\frac{\partial G_a}{\partial y}(x,y)=O(x^{l(a)})$ and $l(a)\geq\min(s(a)-1,r(a),r(a)-1)$, i.e., $l(a)\geq s(a)-1>0$
\item if $\log(x)$ is integrable with respect to the invariant measure $\mu_a$ of $T_a$
then the first return time of $X_a^t$ to $\Sigma$ is integrable.
\end{enumerate}
\end{remark}

\section{Further properties of the one dimensional contracting Lorenz map}\label{sec:one_dimensional}
In \cite{Me00}, conditions (C1) and (C3) were used to prove the existence of
an ergodic absolutely continuous invariant probability measure for Rovella
parameters. In order to obtain uniqueness of that measure, Metzger needed to
consider a slightly smaller class of parameters (still with full density at
0), for which conditions (C2) and (C3) imply a strong mixing property. But,
in \cite{AS2012}, Alves and Soufi deduce the uniqueness of the ergodic
absolutely continuous invariant probability measure for $a\in E$ not assuming
any strong mixing property. Hence, for each $a \in E$, the map $T_a$ has a
unique SRB measure $\mu_a$.

We now recall some recent results of \cite{AS2012} on the statistical
properties of the contracting Lorenz one dimensional maps that we will use in our paper.

To state these statistical properties, we start recalling some definitions and facts about $T_a$, with $a$ a Rovella parameter.

\begin{definition} We say that $T_a$ is \textbf{non-uniformly expanding} if there is a $c> 0$ such that for Lebesgue almost every $x\in I$
\begin{equation}\label{nonuniform}
\lim \inf_{n\to \infty} \frac{1}{n} \sum_{i=0}^{n-1} \log(T_a'(T_a^i(x))) > c.
\end{equation}
\end{definition}

\begin{definition}\label{def:slow_recurrence}
We say that $T_a$ has \textbf{slow recurrence to the critical set} if for every $\epsilon > 0$ there exists $\delta > 0$ such that for Lebesgue almost every $x \in I$ it holds
\begin{equation}\label{slowrec}
\lim\sup_{n\to\infty} \frac{1}{n}\sum_{i=0}^{n-1} -\log d_\delta(T_a^i(x),0)\leq \epsilon,
\end{equation}
where $d_\delta$ is the $\delta$-truncated distance, defined as

\begin{align}
d_\delta(x,y)=
\left\{
\begin{array}{cccc}
      |x -y|,  & \textrm{ if } |x-y| \leq \delta,       \\
      1,  &  \textrm{ if } |x-y| > \delta.
\end{array}
\right.
\end{align}
\end{definition}

\begin{definition}
The \textbf{expansion time function} is defined as
$$
\cE_a(x)=\min\{N\geq 1; \frac{1}{n}\sum_{i=0}^{n-1}
\log T_a'(T_a^i(x)) > c, \forall n \geq N\},
$$
which is well defined and finite almost everywhere in $I$, provided (\ref{nonuniform}) holds almost everywhere.

Fixing $\epsilon > 0$ and choosing $\delta > 0$ conveniently, we define the \textbf{recurrence time function}
$$
\cR_a(x)=\min\{ N\geq 1;  \frac{1}{n}\sum_{i=0}^{n-1}
-\log d_\delta(T_a^i(x),0) < \epsilon, \forall n \geq N\},
$$
which is defined and finite almost every where in $I $, as long as 
(\ref{slowrec}) holds almost everywhere.
\end{definition}

\begin{definition}
We define the \textbf{tail set at time $n$} to be the set of points which at time $n$ have not yet achieved either the uniform exponential growth of the derivative or the uniform slow recurrence:
$$
\Gamma_a^n=\{x\in I; \cE_a(x) > n  \mbox{ or } \cR_a(x) > n\}.
$$
\end{definition}

\begin{theorem}(\cite[Theorem A]{AS2012}) Each $T_a$, with $a \in E$, is non-uniformly expanding and has slow recurrence to the critical set. 
Moreover, there are $C>0$ and $\tau >0$ such that for all $a\in E$ and $n \in \natu$, it holds
\begin{equation}\label{eq:tail}
|\Gamma_a^n|\leq C e^{-\tau n}.
\end{equation}
\end{theorem}

In \cite{AS2012} the authors deduced several interesting consequences from \eqref{eq:tail},
which follow from Theorem 2 and Theorem 3 of the seminal paper \cite{LY99}.
Their results involve the class of H\"older continuos functions\footnote{we will denote by 
\[
\textrm{H\"ol}_{\alpha}(f)=\sup_{x,y\in \bI}\frac{|f(x)-f(y)|}{|x-y|^{\alpha}}
\]
and by $||f||_{H(\alpha)}:=||f||_{\infty}+\textrm{H\"ol}_{\alpha}(f)$ the $\alpha$-H\"older norm.
}
with a given exponent $\alpha>0$, denoted by $H(\alpha)$. 

The main result in \cite{AS2012} is that for all Rovella parameter $a \in E$:
\begin{enumerate}
 \item $T_a$ has a unique ergodic absolutely continuous invariant probability measure $\mu_a$;
 \item the measure $\mu_a$ has exponential decay of correlations for $H(\alpha)$-observables against $L^\infty(\mu_a)$ observables.
\end{enumerate}


A concept strictly related to the decay of correlations is the concept of speed of convergence to equilibrium.
\begin{defi}
\label{def:exp-conv-equil} We say that $(F,\mu )$ has exponential
convergence to equilibrium with respect to norms $\Vert .\Vert _{a}$ and $
\Vert .\Vert _{b}$, if there are $C,\Lambda \in \mathbb{R}^{+},~\Lambda <1$
such that for $f\in L^{1}(m),g\in L^{1}(m)$ and for all $n \geq 1$ it holds
\begin{equation*}
Conv_n(f,g):=\left\vert \int f\cdot (g\circ T^{n})\,dm-\int g\,d\mu \int f\,dm\right\vert
\leq C\Lambda ^{n}\cdot \Vert g\Vert _{a}\cdot \Vert f\Vert _{b}.
\end{equation*}
\end{defi}

From \eqref{eq:tail} and \cite{LY99} it follows that the systems we consider have exponential convergence to equilibrium with respect to $L^\infty$ and H\"older observables (see also \cite[Appendix B]{CCS} for a standard procedure to get a uniform statement for all the obserbables in the given classes).

\begin{prop}\label{pro-convequilibrio} Let $a\in E$, let $\mu_{a}$ be the absolutely continuous invariant measure of $T_{a}$. Then
$(T_{a},\mu_{a} )$ has exponential convergence to equilibrium in the following
sense. There are $C,\Lambda \in \mathbb{R}^{+},~\Lambda <1$ such that for each $f\in H(\alpha)$ and $g\in L^{\infty }(\mu _{a})$ 
\begin{equation*}
Conv_n(f,g)\leq C\Lambda ^{n}\cdot \Vert g\Vert _{\infty }\cdot
||f||_{ \Ha}
\end{equation*}
\end{prop}


%

In the next section we will see how a result on the convergence to
equilibrium for a \ certain map can be extended to a skew product with
contracting fibers, based on that map.

\section{General result on convergence and correlation decay for
skew-products with contracting fibers}
\label{s-general}

We recall some general results from \cite{VGP2014}.
For this, let $\bI=[-1/2,1/2]$ the unit interval and denote by $\QQ=\bI \times \bI$. We  consider maps $F:\QQ\to \QQ$
preserving a regular foliation, which contracts the leaves and whose
quotient map (the induced map on the space of leaves) has exponential
convergence to equilibrium. We will give an estimation of the speed of
convergence to equilibrium for this kind of maps, establishing that it is
also exponential.

\begin{defi}
\label{def:pi-f} If $f:\QQ\rightarrow \QQ$ is integrable, we denote by $%
\pi (f):\bI \rightarrow \bI$ the function $\pi (f):x\mapsto
\int_{\bI}f(x,t)~dt.$
\end{defi}

Let us consider the following anisotropic
norm, considering Lipschitz regularity only on the vertical, $y$, direction. Let $
\Vert \cdot \Vert _{y- \Lip}$ be defined by 
\begin{equation}
\Vert g\Vert _{y- \Lip}=\Vert g\Vert _{sup}+\Lip_{y}(g),
\label{eq:norm-lip}
\end{equation}
where  
\begin{equation}
\Vert g\Vert _{sup}:=\underset{x,y\in\bI}{\sup }|g(x,y)|\quad 
\text{and}\quad \Lip_{y}(g):=\sup_{\substack{ x,y_{1},y_{2}\in \bI 
\\ y_{1}\neq y_{2}}}\frac{|g(x,y_{2})-g(x,y_{1})|}{|y_{2}-y_{1}|}.
\label{eq:lip_y}
\end{equation}

The following is proved in \cite[Theorem 1]{VGP2014} 

\begin{thm}
\label{resuno} Let $F:\QQ\circlearrowleft $ be a Borel function such that $F(x,y)=(T(x),G(x,y))$. 
Let $\mu $ be an $F$-invariant measure with
absolutely continuous marginal $\mu _{T}$ on the $x$-axis which, moreover,
is $T$-invariant. Let us suppose that

\begin{enumerate}
\item $(T,\mu _{T})$ has exponential convergence to equilibrium with respect
to the norm $\Vert \cdot \Vert _{\infty }$ (the $L^{\infty }$ norm) and to a
norm (on the base space) which we denote by $\Vert \cdot \Vert _{Base}$ also suppose that $\Vert \cdot \Vert _{Base}\geq ||\cdot ||_{\infty }$ .

\item $T$ is nonsingular with respect to the Lesbegue measure, piecewise
continuous and monotonic: there is a collection of intervals $
\{I_{i}\}_{i=1,...,m}$, $\cup I_{i}=I$ such that on each $I_{i}$, $T$ is an
homeomorphism onto its image.

\item $F$ is a contraction on each vertical leaf: $G$ is $\lambda $
-Lipschitz in $y$ with $\lambda <1$.
\end{enumerate}

Then $(F,\mu )$ has exponential convergence to equilibrium in the following
sense. There are $C,\Lambda \in \mathbb{R}^{+},~\Lambda <1$ such that 
\begin{equation*}
Conv_n(f,g) \leq C\Lambda ^{n}\cdot \Vert g\Vert _{y- \Lip}\cdot (||\pi
(f)||_{Base}+||f||_{1})
\end{equation*}
for each $f\geq 0$.
\end{thm}

Now let us relate convergence to equilibrium to decay of correlations. This
will be done by the following statement (see \cite[Theorem 2 and 3]{VGP2014})

\begin{thm}
\label{thm:summary-exp-decay} Let $F:\QQ\circlearrowleft $ be a Borel function
such that $F(x,y)=(T(x),G(x,y))$, $\mu $ an $F$-invariant probability
measure with absolutely continuous $T$-invariant marginal $\mu _{T}$ on the $x$-axis and satisfying

\begin{enumerate}
\item $(T,\mu _{T})$ has exponential convergence to equilibrium with respect
to the norm $\Vert \cdot \Vert _{\infty }$ and to a norm $\Vert \cdot \Vert
_{Base}$;

\item $T$ is nonsingular with respect to the Lesbegue measure, piecewise
continuous and monotonic: there is a collection of intervals 
$\{I_{i}\}_{i=1,...,m}$, $\cup I_{i}=\bI$ such that on each $I_{i}$, $T$ is an
homeomorphism onto its image.

\item $F$ is a uniform contraction on each vertical leaf.

\item \label{it:varsquare} Moreover, let us assume that that there are $
C,K\in \mathbb{R}$ and a seminorm $\Vert \cdot \Vert _{\square }$ such that 
\[\Vert \pi (f\circ F^{n})\Vert _{Base}+\Vert f\circ F^{n}\Vert _{\square }\leq
C_{1}K^{n}(\Vert \pi (f)\Vert _{Base}+\Vert f\Vert _{y- \Lip}+\Vert
f\Vert _{\square }),\quad \forall n\geq 1.
\]
\end{enumerate}
Then $F$ has exponential decay of correlations: there are $C,\Lambda \in 
\mathbb{R}^{+},~\Lambda <1$ 
\begin{equation*}
\left\vert \int f\cdot (g\circ F^{n})~d\mu -\int g~d\mu \int f~d\mu
\right\vert \leq C_{2}\Lambda ^{n}\Vert g\Vert _{y- \Lip}(\Vert
f\Vert _{y- \Lip}+\Vert \pi (f)\Vert _{Base}+\Vert f\Vert _{\square })
\end{equation*}
for all $f,g:\QQ\rightarrow \real$ and $n\geq 0$.
\end{thm}


The notations $\Vert \cdot \Vert _{Base}$ and $\Vert \cdot \Vert _{\square }$ emphatize that these are respectively a norm for functions on the base space of the skew product, 
and seminorm for functions on the whole space, the square $\QQ $. In the next section we will find  concrete examples of a norm and  a seminorm with the required properties 
(the $p$-variation norm in the next section and the seminorm $\text{Var}^{\square}$ in the following one). 

\section{Decay of correlations for adapted function spaces}

\label{sec:p-boundeddecay} \label{s-boundvariation} 
Many results about decay
of correlations and convergence to equilibrium are obtained for H\"{o}lder
observables. Yet, in systems with discontinuities, this class of functions
is not  the most natural, since it is not preserved by the transfer operator.
We show how to extend those results to generalized bounded variation
observables. This extension is necessary to apply Theorem
\ref{thm:summary-exp-decay}.

\subsection{Functions of bounded $p$-variation}

We recall the main definitions and basic results about bounded $p-$variation
(see \cite{Ke85}).

Given a function we define its universal $p-$Variation as the following
adaptation of the usual notion of bounded variation.

\begin{defi}
Let $g:[0,1]\rightarrow \mathbb{R}$ and let: 
\begin{equation*}
\text{Var}_{p}(g,x_{1},...,x_{n}) =\left( \sum_{i\leq
n}|g(x_{i})-g(x_{i+1})|^{p}\right) ^{\frac{1}{p}}.
\end{equation*}
The \textbf{universal $p-$Variation} is: 
\begin{equation*}
\text{Var}_{p}(g) =\sup_{\mathcal{P}}\text{Var}_{p}(g,x_{1},...,x_{n}),
\end{equation*}
where $\mathcal{P}$ is the collection of all the finite subdivisions of $[0,1]$. Let 
\begin{equation*}
\textrm{UBV}_{p}=\{g:\text{Var}_{p}(g)<\infty \}
\end{equation*}
be the space of functions of bounded universal $p$-Variation.
\end{defi}

In the following $m$ be the Lebesgue measure on the unit interval, $\epsilon
>0$ and $h:[0,1]\rightarrow \mathbb{C}$. We define 
\begin{equation*}
\text{osc}(h,\epsilon ,x)=\text{ess}\sup
\{|h(y_{1})-h(y_{2})|:y_{1},y_{2}\in B_{\epsilon }(x)\},
\end{equation*}
where $B_{\epsilon }(x)$ is the ball centered in $x$ with radius $\epsilon $
, and the essential supremum is taken with respect to the Lebesgue measure. 
Now let us define 
\begin{equation*}
\text{osc}_{p}(h,\epsilon )=\Vert \text{osc}(h,\epsilon ,x)\Vert _{p},\quad
1\leq p\leq \infty ,
\end{equation*}
where the $p$-norm is taken with respect to $m$.

\begin{remark}
\label{rmk:osc_p} $\text{osc}_{p}(h,\ast ):(0,A]\rightarrow \lbrack 0,\infty
]$ is a non decreasing function and $\text{osc}_{p}(h,\epsilon )\geq \text{
osc}_{1}(h,\epsilon )$.
\end{remark}

Fixed $0\leq r\leq 1$, set $R_{p,r,n}=\{h|\forall \epsilon \in (0,A], \text{
osc}_{p}(h,\epsilon )\leq n\epsilon ^{r}\} $ and $S_{p,r}=\cup _{n\in 
\mathbb{N}}R_{p,r,n}. $

We can now define:

\begin{enumerate}
\item $BV_{p,r}$ as the space of $m-$equivalence classes of functions in $S_{p,r}$

\item $\text{Var}_{p,r}(h)=\sup_{0<\epsilon \leq A} (\epsilon^{-r}\text{osc}
_{p}(h,\epsilon))$ (we remark that this definition depends on a fixed constant $A$ and that $%
\text{Var}_{p,r}(h)\geq \text{Var}_{1,r}(h)$).

\item for $h\in BV_{p,r}$ we define $\|h\|_{p,r}:=\text{Var}_{p,r}(h)+\|h\|_{p}$.
\end{enumerate}

It turns out that $BV_{p,r}$ with the norm $||h||_{p,r}$ is a Banach space;
see \cite[Thm. 1.13]{Ke85}. In the following we will fix $A=1$.

\begin{prop}
\label{cmp}$\textrm{UBV}_{p}\subseteq BV_{p,\frac{1}{p}}\subseteq BV_{1,\frac{1}{p}}$
for all $1\leq p<\infty $. Moreover 
\begin{equation}
\text{Var}_{1,\frac{1}{p}}(h)\leq \text{Var}_{p,\frac{1}{p}}(h)\leq 2^{\frac{%
1}{p}}\text{Var}_{p}(h).  \label{4.6}
\end{equation}
\end{prop}

In what follows we need to compare the $\Vert \cdot \Vert _{p,r}$ norm with
the $L^{\infty }(m)$ norm. The following Lemma will be useful (see \cite[%
Lemma 2]{VGP2014}).

\begin{lemma}
\label{lemaa} If $f\in BV_{1,r}$ ($r\leq 1$), then $f\in L^{\infty }(m)$ and 
\[\Vert f\Vert _{\infty }\leq A^{r-1}\cdot \Vert f\Vert _{1,r},\] where $A$ is
the constant in the definition of $\Vert \cdot \Vert _{1,r}$ (see item 2
above).
\end{lemma}

\subsection{H\"{o}lder convergence to equilibrium implies convergence to
equilibrium for bonded $p$-Variation}

Suppose we have a system having exponential convergence to equilibrium with
$\Ha$ and $L^{\infty }$ observables, let us estimate the convergence for $%
f\in BV_{1,p}$ and $g\in L^{\infty }$.

\begin{prop}
\label{holdertobv}If \ for each $f \in \Ha$ and $g\in L^{\infty }$ we
have convergence to equilibrium with speed $\Phi :$ 
\begin{equation*}
Conv_n(f,g):=\left\vert \int g\circ T^{n}~f~dm-\int g~d\mu \int f~dm\right\vert
 \leq
||g||_{\infty }||f||_{\Ha}\Phi (n),
\end{equation*}%
then for each $f,g$ respectively in $BV_{1,\alpha }$ and $L^{\infty }$ it
holds%
\begin{equation*}
Conv_n(f,g) \leq
||g||_{\infty }||f||_{1,\alpha }6\sqrt{\Phi (n)}.
\end{equation*}
\end{prop}

\begin{proof}
Let us consider $\rho _{\epsilon }=\frac{1}{2\epsilon }1_{B(0,\epsilon )}$ a
multiple of the caracteristic function of an interval of radius $\epsilon$,
small.

Let us consider $f\in BV_{1,\alpha }$, approximate $f$ with $f_{\epsilon
}=f\ast \rho _{\epsilon }$ (the convolution with $\rho _{\epsilon }$) and
estimate the integral:%
\begin{eqnarray*}
\left\vert \int g\circ T^{n}~fdm\right\vert &\leq &\left\vert \int g\circ
T^{n}~(f+f_{\epsilon }-f_{\epsilon })~dm-\int g~d\mu \int f+f_{\epsilon
}-f_{\epsilon }~dm\right\vert \\
&\leq &\int \left\vert g\circ T^{n}~(f-f_{\epsilon })\right\vert dm+\int
g~d\mu \int |f-f_{\epsilon }|~dm \\
&&+|\int g\circ T^{n}~f_{\epsilon }dm-\int g~d\mu \int f_{\epsilon }dm| \\
&\leq &2||g||_{\infty }||f-f_{\epsilon }||_{1}+||g||_{\infty }||f_{\epsilon
}||_{H(\alpha)}\Phi (n).
\end{eqnarray*}

Now let us estimate $||f-f_{\epsilon }||_{1}$%
\begin{eqnarray*}
||f-f\ast \rho _{\epsilon }||_{1} &\leq &\int_{I}|\int_{B(0,\epsilon
)}[f(x-y)-f(x)]\rho _{\epsilon }(y)~dy~|~dx \\
&\leq &\int_{I}|\sup_{y\in B(x,\epsilon )}[f(x-y)-f(x)]|~dx \\
&\leq &\text{osc}_{1}(f,\epsilon ).
\end{eqnarray*}

We bound the H\"older seminorm of $f_{\epsilon }$%
\begin{eqnarray*}
\text{H\"{o}l}_{\alpha }(f\ast \rho _{\epsilon }(x)) 
=\sup_{x_{1},x_{2}\in I}|x_{1}-x_{2}|^{-\alpha }|\int_{B(0,\epsilon
)}[f(x_{1}-y)-f(x_{2}-y)]\rho _{\epsilon }(y)dy|,
\end{eqnarray*}
by the definition of $\rho _{\epsilon }$%
\begin{equation*}
|\int_{B(0,\epsilon )}[f(x_{1}-y)-f(x_{2}-y)]\rho _{\epsilon }(y)dy|\leq
\left\{ 
\begin{array}{c}
2\epsilon ^{-1}|x_{1}-x_{2}|~||f||_{\infty }\,\text{if}\, |x_{1}-x_{2}|\leq
\epsilon \\ 
2||f||_{\infty }\,~\text{if}\,|x_{1}-x_{2}|>\epsilon%
\end{array}%
\right. .
\end{equation*}
Hence%
\begin{equation*}
\text{H\"{o}l}_{\alpha }(f\ast \rho _{\epsilon }(x))\leq
\sup_{x_{1},x_{2}\in I}\left(\left\{ 
\begin{array}{c}
|x_{1}-x_{2}|^{-\alpha +1}2\epsilon ^{-1}||f||_{\infty }\,\text{if}%
\,|x_{1}-x_{2}|\leq \epsilon \\ 
|x_{1}-x_{2}|^{-\alpha }2||f||_{\infty }\,\text{if}\,|x_{1}-x_{2}|\geq
\epsilon .%
\end{array}%
\right. \right).
\end{equation*}
It follows
\begin{equation*}
\text{H\"{o}l}_{\alpha }(f\ast \rho _{\epsilon }(x))\leq 2\epsilon ^{-\alpha
}||f||_{\infty }\leq 2\epsilon ^{-\alpha }||f||_{1,\alpha }.
\end{equation*}

Summarizing

\begin{eqnarray*}
\left\vert \int g\circ T^{n}fdm-\int g~d\mu \int f~dm\right\vert &\leq & \\
&\leq &2||g||_{\infty }||f-f_{\epsilon }||_{1}+||g||_{\infty }||f_{\epsilon
}||_{H(\alpha)}\Phi (n) \\
&\leq &||g||_{\infty }\left( 2\text{osc}_{1}(f,\epsilon )+4\epsilon
^{-\alpha }||f||_{1,\alpha }\Phi (n)\right) \\
&\leq &||g||_{\infty }\left( 2\epsilon ^{\alpha }||f||_{1,\alpha }+4\epsilon
^{-\alpha }||f||_{1,\alpha }\Phi (n)\right) \\
&=&||g||_{\infty }||f||_{1,\alpha }\left( 2\epsilon ^{\alpha }+4\epsilon
^{-\alpha }\Phi (n)\right) .
\end{eqnarray*}

For each $n$ we can take $\epsilon $ such that $\epsilon ^{2\alpha }=\Phi (n)$ we have $%
\epsilon ^{\alpha }=\sqrt{\Phi (n)}$ and%
\begin{equation*}
\left\vert \int g\circ T^{n}fdm-\int g~d\mu \int f~dm\right\vert \leq
6||g||_{\infty }||f||_{1,\alpha }\sqrt{\Phi (n)}.
\end{equation*}
\end{proof}

By this proposition it follows that if a system has exponential convergence to equilibrium with respect to H\"{o}lder observables, it has also exponential convergence with respect to generalized bounded variation ones.

\section{Decay of correlations for the two dimensional contracting Lorenz map: proof of Theorem \ref{teoA} \label{2p1}%
}

In the following section we explain which are the norms involved, in our
case, in the statements of Section \ref{s-general} and we will check that
the $2$-dimensional Rovella maps satisfies the hypothesis of our theorem,
implying the proof of Theorem \ref{teoA}.

As showed in Proposition \ref{pro-convequilibrio} 
the one-dimensional Rovella
map for a Rovella parameter $a\in E$ satisfies exponential convergence to
equilibrium with respect to H\"{o}lder and $L^{\infty }$ observables. By
using the results in Section \ref{sec:p-boundeddecay} and in particular
by Proposition \ref{holdertobv} we have the following.

\begin{clly}
If $a$ is a Rovella parameter, the one dimensional Rovella map satisfies
exponential convergence to equilibrium with respect to the norms $%
||.||_{1,\alpha }$ and $||.||_{\infty }$.
\end{clly}

We will need another definition of variation for maps with two variables.
Similarly to the one dimensional case, if $f:\QQ\rightarrow \mathbb{R}$ and $%
x_{i}\leq x_{2}\leq ...\leq x_{n}$, let us define the variation on the square of $f$ as
\begin{equation*}
\Var^{\square }(f,x_{1},...,x_{n},y_{1},...,y_{n})=\sum_{1\leq i\leq
n}|f(x_{i},y_{i})-f(x_{i+1},y_{i})|.
\end{equation*}%
We then consider the supremum $\Var^{\square
}(f,x_{1},...,x_{n},y_{1},...,y_{n})$ over all subdivisions $x_{i}$ and all
choices of the $y_{i}$ 
\begin{equation*}
\Var^{\square }(f)=\sup_{n}\left( \sup_{(x_{i}\leq x_{2}\leq ...\leq
x_{n})\in \mathcal{I},(y_{i})\in \mathcal{I}}\Var^{\square
}(f,x_{1},...,x_{n},y_{1},...,y_{n})\right) .
\end{equation*}

We want to apply Theorem \ref{thm:summary-exp-decay} using $||\cdot||_{1,\alpha}$
as the norm $||\cdot||_{Base}$ and $\Var^{\square }(\cdot)$ as the seminorm $||\cdot||_{\square}$.
Thus we need to prove that item \ref{it:varsquare} of Theorem \ref{thm:summary-exp-decay} is satisfied for this seminorm; to do so, thanks to \cite[Lemma 16]{VGP2014} the only thing we need to prove now is the following.

\begin{lemma}\label{6.2}
For each $a\in [0,a_0]$ we have that $\text{Var}^{\square }(G_a)<\infty $
\end{lemma}

\begin{proof}
Remembering the definition: 
\begin{equation*}
\text{Var}^{\square }(G_a,x_1,\ldots,x_n,y_1,\ldots,y_n)=\sum_n |
G_a(x_{i},y_i)-G_a(x_{i+1},y_{i})|.
\end{equation*}

By Remark \ref{r-propdeG} we have that $\frac{\partial G_a}{\partial x}$ is bounded for all $x\in I$
and there exists $M$ such that 
\begin{equation*}
\bigg|\frac{\partial G_{a}}{\partial x}\bigg|\leq
M |x|^{s(a)-1},
\end{equation*}%
with $s(a)-1>0$. Moreover, we can observe that, if $x< 0$ and $\tilde{x%
}>0$ we have that
\begin{equation*}
|G(x,y)-G(\tilde{x},y)|<1.
\end{equation*}

Therefore, if $-1/2=x_{1}<\ldots <x_{k}<0<x_{k+1}<\ldots <x_{n}=1/2
$ then
\begin{align*}
\sum_{0}^{k-1}|G(x_{i},y_{i})-G(x_{i+1},y_{i})|&
+|G(x_{k},y_{i})-G(x_{k+1},y_{i})| \\
& +\sum_{k+1}^{n-1}|G(x_{i},y_{i})-G(x_{i+1},y_{i})|\leq  \\
\sum_{0}^{k-1}|\frac{\partial G}{\partial x}(\xi _{i})||x_{i+1}-x_{i}|&
+1+\sum_{k+1}^{n-1}|\frac{\partial G}{\partial x}(\xi
_{i})||x_{i+1}-x_{i}|\leq 1+M,
\end{align*}%
where $\xi _{i}\in \lbrack x_{i},x_{i+1}]$ for $i\neq k$.
\end{proof}

From this, as in \cite{VGP2014}, by Theorems \ref{resuno} and \ref{thm:summary-exp-decay}  follows the decay of
correlation with respect to Lipschitz observables, proving Theorem \ref{teoA}.

\begin{prop}
\label{dec} If $a$ is a Rovella parameter, the two dimensional Rovella map $%
F_{a}$ has exponential convergence to equilibrium and exponential decay of correlations:
There are $C,\Lambda \in \mathbb{R}^{+},~\Lambda <1$ such that for $n\geq 1:$%
\begin{equation*}
Conv_n(f,g)\leq C\Lambda ^{n}\cdot \Vert g\Vert _{y- \Lip}\cdot (||\pi
(f)||_{1,\alpha }+||f||_{1})
\end{equation*}
\begin{equation*}
\left\vert \int f\cdot (g\circ F^{n})\,d\mu -\int g\,d\mu \int f\,d\mu
\right\vert \leq C\Lambda ^{n}\cdot \Vert g\Vert _{y- \Lip}(\Vert
f\Vert _{y- \Lip}+\Vert \pi (f)\Vert _{1,\alpha }+\Var^{\square }f).
\end{equation*}
\end{prop}



\section{Quantitative Recurrence and Logarithm law, proof of Theorem B and C}\label{s.loglaw}

In this section we will show some consequences of the decay of correlations
proved above. We show how these results imply hitting time and quantitative
recurrence estimations.

Let us consider a discrete time dynamical system $(X,T,\mu )$, where $(X,d)$
is a metric space and $T:X\rightarrow X$ is a measurable map preserving a
finite measure $\mu $. Let us consider two points $x,y$ in $X$ and the time
which is necessary for the orbit of $x$ to approach $y$ at a distance less
than $r$%
\begin{equation*}
\tau _{r}(x,y)=\min \{n\in \mathbb{N^{+}}:d(T^{n}(x),y)<r\};
\end{equation*}
if $x=y$ we will use the notation $\tau_{r}(x):=\tau _{r}(x,x)$.

We consider the behavior of $\tau _{r}(x,y)$ as $r\rightarrow 0$. In many
interesting cases this is a power law $\tau _{r}(x,y)\sim r^{R}$. 
When $x\neq y$ the exponent is a quantitative measure
of how fast the orbit starting from $x$ approaches a point $y$. When $x=y$
the exponent $R$ gives a quantitative measure of the speed of recurrence of
an orbit near to its starting point, and this will be a quantitative
recurrence indicator.

\begin{definition}
Let $\nu$ be a measure on $X$, metric space. 
The upper and lower local dimension are defined as:
\[
\overline{d}_{\nu}(x)=\limsup_{r\to 0}\frac{\log(\mu(B_r(x)))}{\log(r)}\quad \underline{d}_{\nu}(x)=\liminf_{r\to 0}\frac{\log(\mu(B_r(x)))}{\log(r)}.
\]
If $\overline{d}_{\nu}(x)=\underline{d}_{\nu}(x)$ we say the local dimension exists at $x$ and denote it by $d_{\nu}(x)=\underline{d}_{\nu}(x)$;
if the local dimension exists and $d_{\nu}(x)$ is constant for $\nu$ a.e. $x$, we say the system is exact dimensional.
\end{definition}

The results of \ \cite{SR} and \cite{Ga07}, give a quantitative estimation
for these indicators for rapidly mixing systems, and can be summarized in
the following theorem (see also \cite{GSR}) that directly implies Theorem B.
\begin{thm}
\label{maine} Let $(X,T,\nu )$ be a measure preserving system with a decay
of correlations with respect to Lipschitz observables faster than any
polynomial rate. Let $x,y\in X,$

\begin{enumerate}
\item if the local dimension $d_{\nu }(y)$ exists then
\begin{equation*}
\limsup_{r\to 0} \frac{\log(\tau_r(x,y))}{-\log(r)}=\liminf_{r\to 0} \frac{\log(\tau_r(x,y))}{-\log(r)}=d_{\nu }(y)  \label{maineq}
\end{equation*}%
for $\nu $-almost every $x$.

\item If $X\subseteq \mathbb{R}^{d}$ for some $d$, then
\begin{equation*}
\limsup_{r\to 0} \frac{\log(\tau_r(x))}{-\log(r)}=\overline{d}_{\nu }(x)\qquad and\qquad \liminf_{r\to 0} \frac{\log(\tau_r(x))}{-\log(r)}=%
\underline{d}_{\nu }(x)
\end{equation*}
for $\nu $-almost every $x$ such that $\underline{d}_{\nu }(x)>0$.
\end{enumerate}
\end{thm}

\begin{proof}[Proof of Theorem B]
Observe that for a Rovella two dimensional map $F_a$, $a\in E$ a Rovella parameter,
the SRB measure $\mu_{F_a}$ has exponential decay of correlations with respect
to Lipschitz observables and that $\underline{d}_{\mu_{F_a}}(x)>0$ almost everywhere 
(the projection of the measure on the basis is absolutely continuous). 
Applying Theorem \ref{maine} we have Theorem B.
\end{proof}

Now we extend the result to the contracting Lorenz flow, following \cite{GN, galapacif09}. 
We need to check that the return time to the section is integrable.

\begin{prop}\label{prop:integrability}
Let $a\in E$ be a Rovella parameter, and let $\mu_a$ be the
invariant measure for the one dimensional map $T_a$ then:
\[
\int_{-1/2}^{1/2} -\log|x| d\mu_a\leq \infty.
\]
\end{prop}
\begin{proof}
In \cite{AS2012} it was proved that for each Rovella parameter
the map $T_a$ has slow recurrence to the critical set (Definition \ref{def:slow_recurrence})
we will use this to give a bound on the integral of $\log|x|$.


First of all, let $I_{\delta}:=(-\delta,\delta)$ and let $J_{\delta}=I\setminus I_{\delta}$:
\[
\int_I-\log|x|d\mu_a=\int_{J_{\delta}}-\log|x|d\mu_a+\int_{I_{\delta}}-\log|x|d\mu_a;
\]
first we will bound
\[
\int_{J_{\delta}}-\log|x|d\mu_a\leq \int_{J_{\delta}}-\log(\delta)d\mu_a\leq \int_I-\log(\delta)d\mu_a=-\log(\delta). 
\]

Now we observe that, if we denote by $\chi_{I_{\delta}}$ the characteristic function of $I_{\delta}$ and 
by $d_{\delta}$ the truncated distance as in Def. \ref{def:slow_recurrence}:
\[
-\log|x|\cdot \chi_{I_{\delta}}=-\log(d_{\delta}(x,0)),
\]
thus
\[
\int_{I_{\delta}}-\log|x|d\mu_a=\int_I-\log|x|\cdot \chi_{I_{\delta}}d\mu_a=\int_I-\log(d_{\delta}(x,0)) d\mu_a.
\]

Let now $\phi_k(x)=\min\{-\log(d_{\delta}(x,0)),k\}$; trivially $||\phi_k||_{\infty}\leq k$, therefore $\phi_k\in L^1(\mu_a)$
and 
\[
\lim_{n\to\infty}\frac{1}{n}\sum_{i=0}^{n-1}\phi_k(T_a^i(x))=\int_I \phi_k d\mu_a,
\]
for $\mu_a$ a.e. $x$. Moreover, for all $x$, $k$ and $n$:
\[
\frac{1}{n}\sum_{i=0}^{n-1}\phi_k(T_a^i(x))\leq \frac{1}{n}\sum_{i=0}^{n-1}-\log(d_{\delta}(T_a^i(x),0)).
\]

We argue by contradiction; suppose $\int_I -\log(d_{\delta}(x,0)) d\mu_a=+\infty$; therefore
\begin{align*}
+\infty&=\lim_{k\to+\infty}\int_I\phi_k(x)d\mu_a=\lim_{k\to+\infty}\lim_{n\to\infty}\frac{1}{n}\sum_{i=0}^{n-1}\phi_k(T_a^i(x))\quad \mu_a-a.e.\\
&\leq \limsup_{n\to\infty}\frac{1}{n}\sum_{i=0}^{n-1}-\log(d_{\delta}(T_a^i(x),0))\leq \epsilon<+\infty
\end{align*}
since the slow recurrence to the critical set holds for a set of full Lebesgue measure.
\end{proof}
\begin{remark}
In particular, since the return time on the section for the geometric contracting Lorenz Flow
is controlled by $-\log|x|/\lambda_1$ this implies that the return time on the section
for the geometric contracting Lorenz Flow is integrable.
\end{remark}
\begin{remark}
Proposition \ref{prop:integrability} is needed in the Rovella case because
the density of $\mu_a$ may be unbounded near $x=0$;
this does not happen in the Lorenz case since we have stronger results
concerning the regularity of the density of the invariant measure of 
the one dimensional map that cannot be applied in the non uniformly hiperbolic case.
\end{remark}

We show now that the return time on the section is integrable for all the Rovella parameters.
\begin{prop}
If the neighborhood $\cU$ is small enough, then the return time to the section is integrable for each Rovella
parameter $a\in E$.
\end{prop}
\begin{proof}
Due to the linearization argument in Section \ref{sec:linearization}, we know that there exist 
two constants $C_1$ and $C_2$ such that the return time to the cross section $\tau_X$ satisfies: 
\[
C_1-\frac{\log|x|}{\tilde{\lambda}_1}\leq \tau_X(x,y)\leq C_2-\frac{\log|x|}{\tilde{\lambda}_1}.
\]
Since $-\log|x|$ is integrable for each Rovella parameter, we have the thesis.
\end{proof}
Consider a flow $X^{t}$ in $\mathbb{R}^{3}$ having a transversal section $%
\Sigma $ whose first return time is integrable. As before let $F:\Sigma
\setminus \Gamma \rightarrow \Sigma $ \ be the first return map associated.
Let $\mu _{F}$ be an ergodic invariant measure for $F$. Now, if we consider
a flow having such a map as its Poincar\'{e} section and integrable return
time, we can construct an SRB invariant measure $\mu _{X}$ for the flow. Let 
$x,x_{0}\in \mathbb{R}^{3}$ and 
\begin{equation*}
\tau _{r}^{X^{t}}(x,x_{0})=\inf \{t\geq 0|X^{t}(x)\in B_{r}(x_{0})\}
\end{equation*}%
be the time needed for the $X$-orbit of a point $x$ to enter for the \emph{%
first time} in a ball $B_{r}(x_{0})$. The number $\tau _{r}^{X^{t}}(x,x_{0})$
is the \emph{hitting time associated to} the flow $X^{t}$ and $B_{r}(x_{0})$.

If the orbit $X^{t}$ starts at $x_{0}$ itself let us consider the return
time in the ball and denote
\begin{equation}
\tau _{r}^{X}(x_{0})=\inf \{t\in \mathbb{R}^{+}:X^{t}(x_{0})\in
B_{r}(x_{0}),\exists i<t,s.t.X^{t}(x_{0})\notin B_{r}(x_{0})\}.  \label{eq-2}
\end{equation}

If $x,x_{0}\in \Sigma $ and $B_{r}^{\Sigma }(x_{0})=B_{r}(x_{0})\cap \Sigma $
we denote 
\begin{equation*}
\tau _{r}^{\Sigma }(x,x_{0})=\min \{n\in N^{+};F^{n}(x)\in B_{r}^{\Sigma
}(x_{0})\}
\end{equation*}%
the \emph{hitting time associated to} the discrete system $F$.

Given any $x$ we recall that we denote with $t(x)$ the first strictly
positive time, such that $X^{t(x)}(x)\in \Sigma $ (the \emph{return time} of 
$x$ to $\Sigma $). A relation between ${\tau _{r}}^{X}(x,x_{0})$ and $\tau
_{r}^{\Sigma }(x,x_{0})$ is proved in \cite{galapacif09} (Proposition 5.2). Let $x\in \mathbb{R%
}^{3}$ and $\pi (x)$ be the projection on $\Sigma $ given by $\pi (x)=y$ if $%
x$ is on the orbit of $y\in \Sigma $ and the orbit from $y$ to $x$ does not
cross $\Sigma $ (if $x\in \Sigma $ then $\pi (x)=x$).

\begin{prop}
\label{c:relacaodimensao} Under the assumptions listed above,  there is a full $\mu _{X}$ measure set $B\subset {%
\mathbb{R}}^{3}$ such that if $x_{0}\in {\mathbb{R}}^{3}$ is regular and $%
x\in B$ it holds (provided the second limits exist) 
\begin{equation}
\lim_{r\rightarrow 0}\frac{\log \tau _{r}^{X^{t}}(x,x_{0})}{-\log r}%
=\lim_{r\rightarrow 0}\frac{\log \tau _{r}^{\Sigma }(\pi (x),\pi (x_{0}))}{%
-\log r}.  \label{eq:tempos}
\end{equation}
\end{prop}

By Theorem B, applying this last proposition to the 2-dimensional
system $(\Sigma ,F_{a},\mu _{a})$ and the contracting geometric Lorenz flow
and its invariant SRB measure, we get the logarithm law and the quantitative
recurrence statement for the flow, i.e., Theorem C.

\section{Appendix: Linearization and properties of the Poincar\'{e} map}\label{sec:linearization}
To study the order of partial derivative of the first return map to the transverse section
we use a $C^1$ linearization near the singularity; we will use Theorem 7.1 page 257 of \cite{Ha02} in its version for flows.

\begin{thm}
Let $n\in\mathbb{Z}^+$ be given. Then there exists an integer $N = N (n)\geq 2$ such that:
if $\Gamma$ is a real non-singular $d\times d$ matrix with eigenvalues $\gamma_1,\ldots,\gamma_d$ satisfying
\[
\sum_{i=1}^d m_i\gamma_i\neq \gamma_k \quad\textrm{for all $k=1,\ldots,d$ and $2\leq\sum_{j=1}^d m_i\leq N(n)$}
\]
and if $\dot{\xi}=\Gamma(\xi) + \Xi(\xi)$ and $\dot{\zeta}= \Gamma\zeta$ , where $\xi,\zeta\in \real^d$ 
and $\Xi$ is of class $C^N$ for small $||\xi||$ with
$\Xi(0) = 0$, $\partial_\xi \Xi(0) = 0$; then there exists a $C^n$ diffeomorphism $R$ from a neighborhood of
$\xi = 0$ to a neighborhood of $\zeta = 0$ such that $R\xi(t)R^{-1}=\zeta(t)$ for all $t \in \real$ and initial conditions
for which the flows $\zeta(t)$ and $\xi(t)$ are defined in the corresponding neighborhood of the origin.
\end{thm}

Since the resonance conditions are open, there exists an $N=N(1)$ such that, if we choose the eigenvalues of the geometric contracting Lorenz Flow
respecting the resonance conditions, there exists a $C^N$ neighborhood of $X_0$ such that all the vector fields 
in the neighborhood respect the resonance conditions and can be $C^1$-linearized.
Generically, the linear part of a vector field $\tilde{X}$ in such a neighborhood is different from the linear part
of $X_0$; our aim is not to find a common linearization for all the fields in the neighborhood but to ensure
the fact that $\tilde{X}$ is $C^1$-linearizable.

\subsection{Behaviour near the fixed point}\label{subsec:near}
Let $\tilde{X}$ be in a $C^N$ neighborhood of $X_0$ such that the resonance condition are still satisfied; 
then $\tilde{X}$ can be $C^1$-linearized.

Denote by $\tilde\lambda_1,\tilde\lambda_2,\tilde\lambda_3$ the eigenvalues of $\tilde{X}$ at the fixed point, and 
denote by $\tilde{r}=-\tilde\lambda_2/\tilde\lambda_1$ and by $\tilde{s}=-\tilde\lambda_3/\tilde\lambda_1$.

First, we will study the behaviour of the flow in a neighborhood of the singularity and then use the information
about the existence of a foliated atlas to obtain informations on the order of derivatives for the first return
Poincar\'{e} maps.

Near the singularity, there exists a coordinate system such that the singularity $p$ is in $0$, the
field is given by $\tilde{X}=(\tilde\lambda_1 x,\tilde\lambda_2 y, \tilde\lambda_3 x)$ and there
exists sections $\tilde\Sigma=\{z=\varepsilon,|x|\leq \varepsilon,|y|\leq \varepsilon\}$, $\tilde{\Sigma}_+=\{x=+\varepsilon,\}$ and $\tilde{\Sigma}_-=\{x=-\varepsilon\}$.

By the same computations as for the geometric contracting Lorenz flow $X_0$ we have that the map from 
$[-\varepsilon,\varepsilon]\times [-\varepsilon,\varepsilon]\subset \tilde\Sigma$ to $\tilde{\Sigma}_+$ is given by:
\[
\tilde F(x,y,1)=(1,\tilde G(x,y),\tilde T(x))=(\varepsilon,\varepsilon^{-\tilde r}y\cdot x^{\tilde{r}},\varepsilon^{1-\tilde s} x^{\tilde{s}}),
\]
and that the time taken between $\tilde{\Sigma}$ and $\tilde{\Sigma}_+$ is given by
\[\tau(x,y,1)=\frac{\log(\varepsilon)-\log(|x|)}{\tilde{\lambda}_1}.\]
Remark that both arguments work also for $\tilde\Sigma_-$.

We can choose the neighborhood such that $\tilde{s}-1>0$ and $\tilde{r}-\tilde{s}>3$.
Therefore, as $x$ approaches $0$ we have that
\[\frac{\partial \tilde{T}}{\partial x}=O(x^{\tilde{s}-1}),\quad \frac{\partial \tilde{G}}{\partial x}=O(x^{\tilde{r}-1}),\quad \frac{\partial \tilde{G}}{\partial y}=O(x^{\tilde{r}}),\]
with $\tilde{s}>1$, $\tilde{r}>1$.

\subsection{Behaviour far from the fixed point}
If $N\geq 3$ we know from \cite{Ro93} that all the vector fields in a neighborhood of $X_0$ preserve a stable foliation.
Let $\tilde{X}$ be a vector field in a $C^N$ neighborhood of $X_0$ such that the resonance conditions are preserved,
and whose flow preserves a stable foliation.
Let $F_{\tilde X}(x,y)=(T_{\tilde{X}}(x),G_{\tilde{X}}(x,y))$ be the first return map of the flow to the section $\Sigma$; we will study the behaviour of
its partial derivatives as $x$ approaches $0$.

With an abuse of notation, we will denote by $\tilde\Sigma$, $\tilde\Sigma_+$ and $\tilde\Sigma_-$
the preimages under the linearizing change of coordinates of the sections $\tilde\Sigma$, $\tilde\Sigma_+$ and $\tilde\Sigma_-$;
the important property is that these preimages are locally $C^1$-manifolds and that they are 
made up of pieces of stable leaves. In particular, we can see $\tilde\Sigma_+$ as a graph of a function of $(y,z)$, where 
the constant leaves are given by constant $z$.

We want to show that, since the flow preserves the stable foliation, then, recalling the results from subsection \ref{subsec:near} we have that 
\begin{enumerate}
 \item the order of $\frac{\partial T}{\partial x}$ as $x$ goes to $0$ is the same as the order of $\frac{\partial \tilde T}{\partial x}$, i.e., $\tilde{s}-1$
 \item the order of $\frac{\partial G}{\partial y}$ is the same as the order of  $\frac{\partial \tilde G}{\partial y}$, i.e., $\tilde{r}$
 \item the order of $\frac{\partial G}{\partial x}$ is at least the minimum of the orders
of $\frac{\partial \tilde T}{\partial x}$, $\frac{\partial \tilde G}{\partial x}$, $\frac{\partial \tilde G}{\partial y}$, i.e., at least $\tilde{s}-1$.
\item if $\log(x)$ is integrable with respect to the invariant measure, the return time to $\Sigma$ is integrable.
\end{enumerate}

Since the flow of $\tilde X$ has no singularities besides $p$, the map $\varphi_2(y,z)=(\mu(z),\nu(y,z))$ that takes $\tilde \Sigma_+$ into $\Sigma$ is 
a diffeomorphism that sends lines $z=\textrm{const}$ into lines $x=\textrm{const}$ and the map $\varphi_1(x,y)=(\chi(x),\zeta(x,y))$ that
takes $\Sigma$ into $\tilde{\Sigma}$ is a map that send lines $x=\textrm{const}$ into lines $x=\textrm{const}$.

By direct computation we have that 
\begin{align*}
DF_{\tilde{X}}&=D\varphi_2\circ D\tilde{F}\circ D\varphi_1\\
&=\left[
\begin{array}{cc}
\frac{\partial \mu}{\partial z}\frac{\partial \tilde T}{\partial x}\frac{\partial \chi}{\partial x} & 0\\
\frac{\partial \nu}{\partial y}\bigg(\frac{\partial \tilde G}{\partial x}\frac{\partial \chi}{\partial x}+\frac{\partial \tilde G}{\partial y}\frac{\partial \zeta}{\partial y}\bigg)+
\frac{\partial \nu}{\partial z}\frac{\partial \tilde T}{\partial x}\frac{\partial \chi}{\partial x} & \frac{\partial \nu}{\partial y}\frac{\partial \tilde G}{\partial y}\frac{\partial \zeta}{\partial y}
\end{array}
\right].
\end{align*}

Since $\varphi_2$ is a diffeomorphism, therefore $\frac{\partial \nu}{\partial y}\cdot \frac{\partial\mu}{\partial z}\neq 0$, and since $\varphi_1$ is a diffeomorphism $\frac{\partial \chi}{\partial x}\cdot \frac{\partial\zeta}{\partial y}\neq 0$.
This proves item (1),(2),(3).

Since the flow has no singularities both the time between $\Sigma$ and $\tilde{\Sigma}$ and the time between $\tilde{\Sigma}_+$ and $\Sigma$ are
bounded; this implies item (4).

\end{document}